\documentclass[eqnumis, pdf]{hjms}
\usepackage{amsfonts}
\usepackage{amsmath}
\usepackage{footnote}
\usepackage{multicol}
\usepackage{booktabs}
\usepackage[flushleft]{threeparttable}
\usepackage[font=small,labelfont=bf]{caption}

\def\R#1{\ensuremath{\mathbb{R}^{#1}}}
\def\S#1{\ensuremath{\mathbb{S}^{#1}}}
\def\<{\left<}
\def\>{\right>}
\def\Xfrak{\mathfrak{X}}
\DeclareMathOperator{\Con}{Con}
\def\bfn{\mathbf{n}}
\def\norma#1{|\kern-.3ex|#1|\kern-.3ex|}
\def\e{\varepsilon}
\def\Y{Z}
\DeclareMathOperator{\arccot}{arccot}
\def\tgamma{\widetilde{\gamma}}

\begin{document}
%please do not change this part%%%%%
\hinfo{XX}{x}{XXXX}{1}{\lpage}{10.15672/hujms.xx}{Article Type}
%%%%%%%%%%%%%%%%%%%%%%%%%%%%%%
%\giris
%{ } %yazarlar
%{ } %başlık
%{ } %ilk sayfa no
%author, title, first page number
%%%%%%%%%%%%%%%%%%%%%%%%%%%%%%
%Corresponding Author Email:
%
%%%%%%%%%%%%%%%%%%%%%%%%%%%%%%

\markboth{P. Lucas, J.A. Ortega-Yagües}{Concircular helices and concircular surfaces in Euclidean 3-space $\R3$}

\title{Concircular helices and concircular surfaces in Euclidean 3-space $\R3$}

\author{Pascual Lucas\coraut$^{1}$, José-Antonio Ortega-Yagües$^1$}

\address{$^1$Departamento de Matemáticas, Universidad de Murcia\\
       Campus de Espinardo, 30100 Murcia SPAIN}
\emails{plucas@um.es (P. Lucas), yagues1974@hotmail.com (J.A. Ortega-Yagües)}
\maketitle
\begin{abstract}
In this paper we characterize concircular helices in $\R3$ by means of a differential equation involving their curvature and torsion. We find a full description of concircular surfaces in $\R3$ as a special family of ruled surfaces, and we show that $M\subset\R{3}$ is a proper concircular surface if and only if either $M$ is parallel to a conical surface or $M$ is the normal surface to a spherical curve. Finally, we characterize the concircular helices as geodesics of concircular surfaces.
\end{abstract}% the abstract
  % AMS subject classifications
        % Keywords

%please do not change this part%%%%%%%%%%
\hinfoo{DD.MM.YYYY}{DD.MM.YYYY} %receive, accept
%%%%%%%%%%%%%%%%%%%%%%%%%%%%%%%%%%%

%Accepted Date:
%Received Date:
%DOI:

\section{Introduction}
\label{s:intro}

The study of curves and surfaces in Euclidean 3-space $\R3$ satisfying a certain condition with respect to a special vector field has a long history. We will not go into detail on this issue now, but we would like to show some classic examples. Generalized helices are defined by the property that their tangents make a constant angle with a fixed direction $v$. A classical result (see \cite{Sai46}) is that a unit speed curve $\gamma$ is a generalized helix if and only if $\tau_\gamma/\kappa_\gamma$ is constant, where $\kappa_\gamma$ and $\tau_\gamma$ stand for the curvature and torsion. Generalized helices have a nice characterization as geodesics of cylinders $M$ whose rulings are parallel to $v$, so that we have $\<N,v\>=0$, $N$ being the unit vector field normal to $M$.

In a similar way, slant helices are defined in \cite{IT1} by the property that their principal normals make a constant angle with a fixed direction $v$. Observe that principal normal lines of a generalized helix are perpendicular to a fixed direction, so that a generalized helix is also a slant helix. It is shown in \cite{IT1} that $\gamma$ is a slant helix if and only if
\[
\frac{\kappa_\gamma^2}{(\kappa_\gamma^2+\tau_\gamma^2)^{3/2}}\left(\frac{\tau_\gamma}{\kappa_\gamma}\right)'=\text{const}.
\]
It is not difficult to see that geodesics of a helix surface $M$ (see \cite{DSRH09}) must be slant helices, and recently it is shown in \cite{LO16a} that every slant helix is a geodesic in a helix surface. In this case, the unit normal vector field $N$ satisfies $\<N,v\>=\text{const}$, where the constant is zero in case the helix surface is a cylinder.

As a third example, rectifying curves in $\R3$ are defined in \cite{Chen03} as the curves whose position vector always lies in its rectifying plane. A necessary and sufficient condition for a curve $\gamma$ with $\kappa_\gamma>0$ be a rectifying curve is that $(\tau_\gamma/\kappa_\gamma)(s)=c_1s+c_0$, for constants $c_0$ and $c_1\neq0$, $s$ being the arclength parameter. It is not difficult to see that geodesics of a conical surface are rectifying curves, and in \cite{Chen17b} the author has shown the converse.

In this paper we want to put the above three examples into a common framework, and to make this possible we need to remember the following concept.
A vector field $Y\in\Xfrak(M)$ on a Riemannian manifold $M$, with Levi-Civita connection $\nabla$, is said to be \emph{concircular} if $\nabla Y=\mu I$, where $\mu\in\mathcal{C}^\infty(M)$ is a differentiable function called the \emph{concircular factor}, \cite{Fia39, Yan40, Kim82, Cram07}. Let $\Con(M)$ denote the set of concircular vector fields on $M$.
Inspired by \cite{DSRH09, Chen16}, we give the following definition.
\begin{definition}\label{concsubm}
Given a submanifold $M\subset\R{n}$ and a concircular vector field $Y\in\Con(\R{n})$, we say that $M$ is a \emph{concircular submanifold} (with axis $Y$) if $\<\bfn,Y\>$ is a constant function along $M$, $\bfn$ being any unit vector field in the first normal space of $M$.
\end{definition}

It is well known (see \cite{Cram07}) that $\Con(\R{n})$ is a real vector space of dimension $n+1$, and a basis is given by the position vector field and $n$ linearly independent constant vector fields. As a consequence, a concircular vector field $Y\in\Con(\R{n})$ is defined by $Y(p)=\mu p+v$, where $\mu\in\R{}$ and $v\in\R{n}$.

This paper is organized as follows. Section \ref{s:helices} contains a characterization of all concircular helices, see theorem \ref{teoconc1}. In Section \ref{s:surfaces} we present some properties of concircular surfaces, see propositions \ref{propcs1} and \ref{propcs2}. We finish that section with some characterizations of concircular surfaces in $\R3$, see theorems \ref{sctheo} and \ref{sctheo2}. Section \ref{s:geodesics} contains the characterization of geodesics curves of concircular surfaces, see propositions \ref{intp1} and \ref{p3}. That characterization result is used to show that concircular helices can be described as the geodesics of concircular surfaces, see theorem \ref{th1}.

\section{Concircular helices}
\label{s:helices}

Let $\gamma:I\to\R3$ be a differentiable curve parametrized by the arclength parameter $s$. At each point of $\gamma$ where $\gamma'(s)\times\gamma''(s)\neq0$, the Frenet frame $\{T_\gamma(s)=\gamma'(s),N_\gamma(s),B_\gamma(s)=T_\gamma(s)\times N_\gamma(s)\}$ satisfies the usual Frenet-Serret equations:
\begin{equation}
T_\gamma'(s)=\kappa_\gamma(s)\,N_\gamma(s),\
N_\gamma'(s)=-\kappa_\gamma(s)\,T_\gamma(s)+\tau_\gamma(s)\,B_\gamma(s),\
B_\gamma'(s)=-\tau_\gamma(s)\,N_\gamma(s).\label{FS-eq}
\end{equation}
The functions $\kappa_\gamma$ and $\tau_\gamma$ are called the curvature and torsion of the curve $\gamma$.

Given a (non geodesic) unit speed curve $\gamma$ in $\R{3}$ and a concircular vector field $Y\in\Con(\R{3})$, from Definition \ref{concsubm} $\gamma$ is a \emph{concircular helix} (with axis $Y$) if $\<N_\gamma,Y\>$ is a constant function along $\gamma$, $N_\gamma$ being the principal normal vector field of $\gamma$.

\begin{example}[Generalized helix]
Generalized helices are concircular helices with axis $Y=v$, and its natural equation is that the ratio $\rho=\tau_\gamma/\kappa_\gamma$ of torsion and curvature is constant (see \cite{Sai46}), i.e. $\rho'=0$.
\end{example}

\begin{example}[Slant helix]
Slant helices are concircular helices with axis $Y=v$, and its natural equation can be rewritten in terms of $\rho=\tau_\gamma/\kappa_\gamma$ as
\begin{equation*}
\frac{\rho'}{\kappa_\gamma(1+\rho^2)^{3/2}}=\text{const}.
\end{equation*}
\end{example}

\begin{example}[Rectifying curve]
Rectifying curves are concircular helices whose axis $Y$ is the position vector field, and its natural equation (see \cite{Chen03}) is
$\rho'=\text{const}\neq0$.
\end{example}

We call $\gamma$ a \emph{proper concircular helix} if $\gamma$ is neither a generalized helix nor a slant helix nor a rectifying curve; in other words, a proper concircular helix is a concircular helix with $\lambda\mu\neq0$, $\lambda$ being the constant $\<N_\gamma,Y\>$.

Let $\gamma\subset\R3$ be a proper concircular helix with associated concircular vector field given by $Y(p)=\mu p+v$, and assume $\<N_\gamma,Y\>=\lambda\in\R{}-\{0\}$. Without loss of generality we can assume that $\gamma$ is parametrized by its arclength parameter $s$. Then
\begin{equation}\label{YRef}
Y(\gamma(s))=a(s)\,T_\gamma(s)+\lambda\,N_\gamma(s)+b(s)\,B_\gamma(s),
\end{equation}
for certain differentiable functions $a$ and $b$. In the following computations we will drop the parameter $s$ to simplify the writing. By taking derivative in (\ref{YRef}) we get
\[
\mu\,T_\gamma=(a'-\lambda\kappa_\gamma)\,T_\gamma+ (a\kappa_\gamma-b\tau_\gamma)\,N_\gamma+ (b'+\lambda\tau_\gamma)\,B_\gamma,
\]
and so
\begin{equation}\label{conc2}
-\lambda\rho\kappa_\gamma=\lambda\left(\frac{\kappa_\gamma(1+\rho^2)}{\rho'}\right)'+ \mu\left(\frac{1}{\rho'}\right)',
\end{equation}
which can be rewritten as
\begin{equation}\label{conc3}
\rho\kappa_\gamma+\left(\frac{\kappa_\gamma(1+\rho^2)}{\rho'}\right)'= \frac{\mu}{\lambda}\frac{\rho''}{\rho'^2}.
\end{equation}
Now, a straightforward computation yields
\begin{equation}\label{conc4}
\left(\frac{\rho'}{\kappa_\gamma(1+\rho^2)^{3/2}}\right)'= m\,\frac{\rho''}{\kappa_\gamma^2(1+\rho^2)^{5/2}},
\end{equation}
where $m=-\mu/\lambda$ is a nonzero constant. Note that equations (\ref{conc2}), (\ref{conc3}) and (\ref{conc4}) are equivalent to each other. Observe that (\ref{conc4}) reduces to classical equation for slant helices when the constant $m$ vanishes.

In the following we will show that equation (\ref{conc4}) characterizes proper concircular helices.

Let $\gamma$ be an arclength parametrized curve whose curvature $\kappa_\gamma$ and function $\rho$ (with $\rho'\neq0$) satisfy equation (\ref{conc4}) for a nonzero constant $m$. Consider the vector field $V(s)$ along $\gamma$ given by
\begin{equation}\label{Darboux}
V=b\,D_\gamma+ \lambda N_\gamma,
\end{equation}
where $D_\gamma=\rho T_\gamma+B_\gamma$ is the (modified) Darboux vector field and $b$ is the differentiable function given by
\begin{equation}\label{ecub}
b=\frac{1}{\rho'}\big(\mu+\lambda\kappa_\gamma(1+\rho^2)\big),
\end{equation}
$\lambda$ and $\mu$ being two nonzero constants such that $\mu=-\lambda m$. Observe that (\ref{conc2}) implies that $b'=-\lambda\tau_\gamma$.
By taking derivative in (\ref{Darboux}) and using (\ref{ecub}) we get
\[
V'=\big(-\lambda\kappa_\gamma(1+\rho^2)+b\rho'\big)\,T_\gamma=\mu\,T_\gamma.
\]
By integrating this equation along $\gamma$ we obtain $V(s)=\mu\gamma(s)+v$, for a constant vector $v$. Finally, we can easily check that the concircular vector field $Y(p)=\mu p+v$ satisfies $\<N_\gamma,Y\>=\lambda$ along $\gamma$, so that $\gamma$ is a proper concircular helix.

We have shown the following characterization result.

\begin{theorem}\label{teoconc1}
Let $\gamma$ be an arclength parametrized curve, with $\kappa_\gamma>0$ and $\rho'\neq0$. $\gamma$ is a proper concircular helix if and only if its curvature $\kappa_\gamma$ and function $\rho=\tau_\gamma/\kappa_\gamma$ satisfy the following differential equation:
\[
\left(\frac{\rho'}{\kappa_\gamma(1+\rho^2)^{3/2}}\right)'= m\,\frac{\rho''}{\kappa_\gamma^2(1+\rho^2)^{5/2}},
\]
for a certain nonzero constant $m\in\R{}$. Moreover, a concircular vector field $Y$ for the proper concircular helix $\gamma$ is the extension of the vector field  $V=b\,D_\gamma+ \lambda N_\gamma$, where $b$ is the differentiable function given by (\ref{ecub}).
\end{theorem}

\section{Concircular surfaces}
\label{s:surfaces}

In the particular case of a surface, $M$ is said to be a concircular surface (with axis $Y$) if $\<N,Y\>$ is a constant function along $M$, $N$ being a unit normal vector field.

Trivial examples of concircular surfaces are planes and spheres.
It should be pointed out that concircular surfaces are different from constant slope surfaces, \cite{Mun10}. A constant slope surface is a surface whose unit normal vector field $N$ makes a constant angle with the position vector. Therefore, the sphere is the only surface that is at the same time a concircular surface and a constant slope surface. We can thus say that these two families are essentially distinct.
Other nontrivial examples of concircular surfaces are the following.

\begin{example}[Generalized cylinder]\label{ex-gc}
Let $M$ be a surface parametrized by
$X(s,z)=\beta(s)+z\,v$,
where $\beta$ is a plane curve with Frenet frame $\{T_\beta,N_\beta\}$ and $v$ is a constant vector orthogonal to the plane containing the curve $\beta$. The unit normal vector field is given $N(s,z)=N_\beta(s)$ and the concircular vector field can be chosen as $Y=v$, satisfying $\<N,Y\>=0$ everywhere.
\end{example}

\begin{example}[Helix surface]\label{ex-hs}
Let $M$ be a surface parametrized by
$X(s,z)=\beta(s)+z\,(\cos\varphi\,N_\beta(s)+\sin\varphi\,v),\ \varphi\in (0,\pi/2]$,
where $\beta$ and $v$ are as in the previous example. Here, the normal vector field is given $N(s,z)=-\sin\varphi\,N_\beta(s)+\cos\varphi\,v$ and the concircular vector field can be chosen again as $Y=v$, satisfying $\<N,Y\>=\cos\varphi$. These surfaces are also known as constant angle surfaces, \cite{DSRH09}.
\end{example}

\begin{example}[Conical surface]\label{ex-cs}
Let $M$ be a surface parametrized by
$X(s,z)=z\,\beta(s)$,
where $\beta$ is a spherical curve, $\beta(s)\in\S{2}(1)$, with Darboux frame $\{T_\beta,N_\beta\}$. In this case, the normal vector field is given by $N(s,z)=N_\beta(s)$ and the concircular vector field $Y$ can be chosen as the position vector field.
\end{example}

A nontrivial concircular surface $M\subset\R{3}$ is said to be \emph{proper} if it is neither a generalized cylinder, nor a helix surface nor a conical surface, i.e., $\lambda\mu\neq0$.

For a proper concircular surface, we can think that the concircular vector field $Y$ is essentially the position vector ($\mu=1, v=0$), so that condition $\<Y(s,z),N(s,z)\>=\lambda$ is equivalent to saying that the distance from the origin to the tangent planes to the surface is constant and equal to $\lambda$, i.e., that all planes tangent to the surface are also tangent to the sphere of radius $\lambda$ centered at the origin.

\begin{proposition}\label{propcs1}
A surface $M\subset\R{3}$ admits a concircular vector field parallel to its normal vector field along $M$ if and only if $M$ is an open piece of a plane or a sphere (i.e., if and only if $M$ is a trivial concircular surface).
\end{proposition}
\begin{proof}
Let $N$ denote the unit normal vector field of the surface and assume there is a concircular vector field $Y$ such that its restriction to $M$ satisfies $Y|_M=\lambda\,N$, for a certain differentiable function $\lambda$. Then we have
$\mu\,X=X(\lambda)\,N-\lambda\,AX$,
for all tangent vector field $X$, where $A$ stands for the Weingarten operator of $M$. By equating the normal components of both sides we deduce that $\lambda$ is a nonzero constant. Moreover we get
$AX=-(\mu/\lambda)\,X$,
showing that $M$ is a totally umbilical surface. This concludes the proof.
\end{proof}

Let $M\subset\R{3}$ be a nontrivial concircular surface and let $Y$ denote the corresponding concircular vector field on $\R{3}$ with $\<Y,N\>=\lambda$ (constant) along $M$. By taking tangent and normal components of $Y$ restricted to $M$ we can write
\begin{equation}\label{scT}
Y|_M=Y^T+Y^N=\delta\,T+\lambda\,N,
\end{equation}
where $T$ is a unit tangent vector field and $\delta$ is a nonzero differentiable function (otherwise, $M$ should be a trivial concircular surface by the previous proposition).
\begin{proposition}\label{propcs2}
Let $M$ be a proper concircular surface in $\R3$. Then:\vspace*{-\topsep}
\begin{enumerate}\itemsep0pt
\item[$i)$] $M$ is a flat surface.
\item[$ii)$] The integral curves of $T$ are straight lines.
\end{enumerate}
\end{proposition}
\begin{proof}
By derivating in (\ref{scT}) with respect to a tangent vector field $X$ and using the Gauss and Weingarten formulae we have
\begin{equation}\label{sceq1}
\mu X=X(\delta)T+\delta\nabla_XT-\lambda AX+\delta\sigma(X,T),
\end{equation}
where $\nabla$ and $\sigma$ denote de Levi-Civita connection and second fundamental form of $M$. By equating normal components of (\ref{sceq1}) and bearing in mind that $\delta\neq0$ we get
$\sigma(X,T)=0$,
for any tangent vector field $X$. These equation is equivalent to
$AT=0$, which implies $i)$.
Now, by taking $X=T$ in (\ref{sceq1}), %and using (\ref{scsigma}) and (\ref{scAT}),
we also obtain
$\nabla_TT=0$,
showing $ii)$.
\end{proof}
\medskip

In the following, we present a method for constructing nontrivial concircular surfaces from curves on totally umbilical surfaces.
Let $S\subset\R3$ be a totally umbilical surface of constant curvature $c$ (we can take $S$ as a plane or a sphere of radius $r$ centered at $p_0$). Take a unit speed curve $\beta:I\subset\R{}\to S\subset\R3$, with curvature $\kappa_\beta(t)$ and Darboux frame $\{T_\beta(t),N_\beta(t),\eta(t)=T_\beta(t)\times N_\beta(t)\}$. Then we have the Frenet-Darboux equations of $\beta$,
\[
T_\beta'(t) = -\sqrt{c}\,\eta(t)+\kappa_\beta(t)\,N_\beta(t),\quad
N_\beta'(t) = -\kappa_\beta(t)\,T_\beta(t),\quad
   \eta'(t) = \sqrt{c}\,T_\beta(t).
\]
Let $M=M_{\beta,\varphi}$ be the ruled surface built on the curve $\beta$ parametrized by
\begin{equation}\label{intparamet}
X(t,z)=\beta(t)+z\,\big(\cos\varphi\,N_\beta(t)+\sin\varphi\,\eta(t)\big),
\end{equation}
for a constant $\varphi\in(0,\pi/2]$.
The unit normal vector field is given by
$N(t,z)=-\sin\varphi\, N_\beta(t)+\cos\varphi\,\eta(t)$,
and this shows that surfaces $M$ parametrized by (\ref{intparamet}) are concircular surfaces. Indeed, if $S$ is a plane with unit normal vector $\eta_0$, then the concircular vector field $Y(p)=\eta_0$ satisfies $\<Y,N\>=\cos\varphi$ constant. Otherwise, if $S$ is a sphere of radius $r$ centered at $p_0$, then the concircular vector field $Y(p)=\sqrt{c}\,(p-p_0)$ also satisfies the condition $\<Y,N\>=\cos\varphi$ constant.

\begin{theorem}\label{sctheo}
Let $M$ be a nontrivial concircular surface with associated concircular vector field $Y$. Then $M$ can be locally parametrized by
\begin{equation}\label{teoparametrization}
X(t,z)=\beta(t)+z\,\big(\cos\varphi\,N_\beta(t)+\sin\varphi\,\eta(t)\big),
\end{equation}
with $\varphi\in(0,\pi/2]$, $\beta$ being a curve in a totally umbilical surface $S\subset\R{3}$ whose normal vector field $\eta$ is parallel to $Y$ along $S$.
\end{theorem}
\begin{proof}
Assume that $Y(q)=\mu\,q+v$, for all point $q\in\R3$, and $\<Y,N\>=\lambda$ for a certain constant $\lambda$, $N$ being the unit vector field normal to $M$. Given a point $p\in M$, with $Y(p)\neq0$, let $S_p$ be a totally umbilical surface containing the point $p$ and whose normal vector field is parallel to $Y$. This surface $S_p$ can be constructed as follows:
\begin{itemize}
\item when $\mu=0$, let $S_p$ be the plane $S_p=\{q\in\R3:\<q-p,v\>=0\}$.
\item when $\mu\neq0$, let $S_p$ be the sphere $\S{2}((-1/\mu)v,r)$ centered at the critical point of vector field $Y$ and of radius $r=\norma{p+(1/\mu) v}$, its distance to the point $p$.
\end{itemize}
From Proposition \ref{propcs1}, there is a point $p\in M$ such that the tangent planes at $p$ to the surfaces $S_p$ and $M$ are distinct. This allows us to define the curve $\beta=S_p\cap M$, parametrized by $\beta(t)$ with $\beta(0)=p$. As before, consider the decomposition of vector field $Y$ along $M$ as $Y=\delta\,T+\lambda\,N$, $T$ being a unit tangent vector field. It is easy to see that the tangent plane to $M$ along $\beta$ is spanned by  $\{T_\beta,\,T\}$ (note that $T$ and $T_\beta$ are orthogonal). By using Proposition \ref{propcs2}, we deduce there exists a neighborhood $U$ of the point $p$ given by
\[
U=\{\beta(t)+zT(\beta(t)):t\in(-\e_1,\e_1),\ z\in(-\e_2,\e_2)\},\ \e_1,\e_2>0.
\]
Since the vector fields $T$ and $T_\beta$ are orthogonal along $\beta$, there is a function $\varphi=\varphi(t)$ such that
\(
T(\beta(t))=\cos\varphi(t)\,N_\beta(t)+\sin\varphi(t)\,\eta(t),
\)
and so we also have
\(
N(\beta(t))=-\sin\varphi(t)\,N_\beta(t)+\cos\varphi(t)\,\eta(t).
\)
Now we can write the concircular vector field $Y$ along $\beta$  as follows
\begin{equation}\label{Ybeta}
Y(\beta(t))= \big(\delta(\beta(t))\cos\varphi(t)-\lambda\sin\varphi(t)\big)N_\beta(t)+
\big(\delta(\beta(t))\sin\varphi(t)+\lambda\cos\varphi(t)\big)\eta(t).
\end{equation}
Since the normal vector field $\eta$ of $S_p$ is parallel to vector field $Y$ and $Y|_\beta$ is of constant length (specifically $\norma{Y|_\beta}=\norma{v}$ when $\mu=0$, and $\norma{Y|_\beta}=|\mu|\,r$ when $\mu\neq0$), then from (\ref{Ybeta}) we get $\delta$ is constant along $\beta$, and so $\varphi$ is constant. Without loss of generality we can assume  $\varphi\in(0,\pi/2]$. Therefore $U$ can be parametrized as in (\ref{teoparametrization}) and this concludes the proof.
\end{proof}

The above result characterizes the three classic examples of concircular surfaces: the cylindrical surfaces ($\mu=0$, $\varphi=\pi/2$), the helix surfaces ($\mu=0$, $\varphi\in(0,\pi/2)$) and the conical surfaces ($\mu\neq0$, $\varphi=\pi/2$).

To finish this section, we present a nice characterization of proper concircular surfaces.
\begin{theorem}\label{sctheo2}
$M\subset\R3$ is a proper concircular surface if and only if one of the following conditions hold:\vspace*{-\topsep}
\begin{enumerate}\itemsep0pt
\item[$1)$] $M$ is the normal surface to a spherical curve.
\item[$2)$] $M$ is parallel to a conical surface.
\end{enumerate}
\end{theorem}
\begin{proof}
Let $M$ be a concircular surface and consider $X(t,z)$ its parametrization as in (\ref{teoparametrization}), where we assume $\beta(t)$ is a unit speed curve in the sphere $\S2(r)$. Take the curve
\(
\delta(t)=r(\sin\varphi\, N_\beta(t)-\cos\varphi\, \eta(t)),
\)
with $\eta(t)=\frac1r\beta(t)$,
whose Darboux frame $\{T_\delta,N_\delta,(1/r)\delta\}$
is given without loss of generality by
$T_\delta(u(t))=-T_\beta(t)$ and
$N_\delta(u(t))=\cos\varphi\, N_\beta(t)+\sin\varphi\,\eta(t)$,
where $u=u(t)$ denotes the arc parameter of $\delta$. From here we obtain
\[
\beta(t)=-\cos\varphi\,\delta(u(t))+r\sin\varphi\,N_\delta(u(t)),\quad
N_\beta(t)=\frac1r\sin\varphi\,\delta(u(t))+\cos\varphi\,N_\delta(u(t)),
\]
and then $X(t,z)$ can be rewritten as
\[
X(t,z)=-\cos\varphi\,\delta(u(t))+(r\sin\varphi+z)\,N_\delta(u(t)).
\]
This shows that $M$ is the normal surface to a spherical curve.\par
In a similar way, using the same type of reasoning but applying it to the curve
\(
\delta(t)=r(\cos\varphi\, N_\beta(t)+\sin\varphi\, \eta(t)),
\)
we show that $M$ is parallel to a conical surface.\par
To prove the converse, let us assume $M$ is the normal surface to a spherical curve  $\delta(u)$ in the sphere $\S2(r)$, with Darboux frame $\{T_\delta,N_\delta,(1/r)\delta\}$, and consider the surface $M$ parametrized by
\[
X(u,z)=\delta(u)+zN_\delta(u).
\]
Then the unit normal vector field $N(u,z)$ is given without loss of generality by $N(u,z)=(1/r)\delta(u)$, and therefore $\<N(u,z),X(u,z)\>=r$, i.e., $M$ is a concircular surface.\par
A similar reasoning can be done if $M$ is parallel to a conical surface.
\end{proof}

\begin{example}[A family of concircular surfaces]\label{ex-sup-conc}
Without loss of generality (see \cite{Sco95, Str88}), an arclength parametrized spherical generalized helix is given by
\begin{align*}
\delta(u)= & \left(\frac{a+m}{2a}\cos\Big(\frac{a-m}{a}\eta(u)\Big)- \frac{a-m}{2a}\cos\Big(\frac{a+m}{a}\eta(u)\Big)\right.,\\
& \kern10pt\frac{a+m}{2a}\sin\Big(\frac{a-m}{a}\eta(u)\Big)- \frac{a-m}{2a}\sin\Big(\frac{a+m}{a}\eta(u)\Big),
\left.
\frac{w}{a}\cos\Big(\frac{m}{a}\eta(u)\Big)\right),
\end{align*}
where $\eta(u)=\frac{a}{m}\arccos(\frac{-m}{w}u)$ and $a^2=m^2+w^2$.
Taking $\mu(u)=\sqrt{w^2-m^2u^2}$, a straightforward computation leads to
\begin{align*}
\delta(u)&= \frac{1}{aw}\Big(-m^2u\cos\eta(u)+a\mu(u)\sin\eta(u), -a\mu(u)\cos\eta(u)-m^2u\sin\eta(u), -mwu\Big),\\
T_\delta(u)&= \frac1a \Big(w\cos\eta(u),w\sin\eta(u),-m\Big),\\
N_\delta(u)&= \frac{1}{aw}\Big(m\mu(u)\cos\eta(u)+amu\sin\eta(u), -amu\cos\eta(u)+m\mu(u)\sin\eta(u), w\mu(u)\Big).
\end{align*}
According to Theorem \ref{sctheo2}, a family of concircular surfaces is given by the following parametrization
\begin{align*}
X(u,v)= & \frac{1}{aw}\Big((-m^2u+m\mu(u)v)\cos\eta(u)+ (a\mu(u)+amuv)\sin\eta(u),\\
   & \kern25pt -a(muv+\mu(u))\cos\eta(u)+ (m\mu(u)v-m^2u)\sin\eta(u),
-mwu+w\mu(u)v\Big),
\end{align*}
for certain non-zero parameters $a,m,w$.
\begin{center}
\includegraphics[width=.4\linewidth,angle=90]{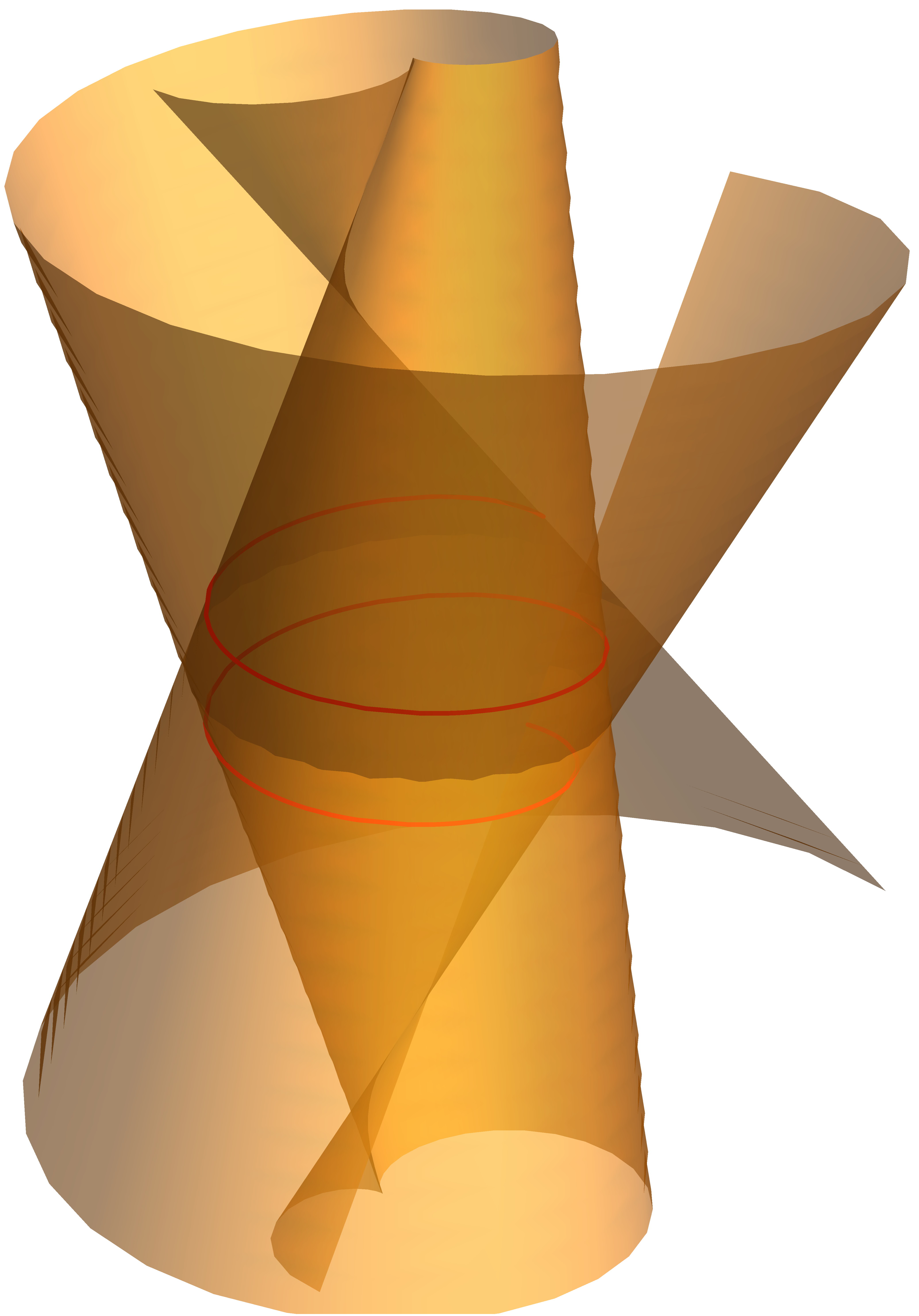}
\end{center}
\end{example}

\section{Geodesics of concircular surfaces}
\label{s:geodesics}

Let $M$ be a nontrivial proper concircular surface. Since any geodesic of $M$ satisfies that its principal normal vector field is parallel to the unit vector field normal to $M$, it is clear that every geodesic on a concircular surface is a concircular helix. Now we will find the equations characterizing the geodesics in a concircular surface.

Let $\gamma(s)=X\big(t(s),z(s)\big)$ be a unit speed geodesic of $M$, with $\kappa_\gamma>0$, where $X$ is given by (\ref{teoparametrization}). Making the same computations as in \cite{LO16a} we can deduce there is a differentiable function $\theta$ such that
\begin{align}
t'(s)\big(1+z(s)(-\cos\varphi\,\kappa_\beta(t(s))+\sqrt{c}\,\sin\varphi)\big) &= \sin\theta(s),\label{p1.1}\\
z'(s) &= \cos\theta(s).\label{p1.2}
\end{align}
From the Frenet equations (\ref{FS-eq}), and using that $\gamma$ is a geodesic, we deduce
\begin{align}
\cos\theta(s)\,\big(\theta'(s)-\cos\varphi\, t'(s)\,\kappa_\beta(t(s))+\sqrt{c}\,\sin\varphi\,t'(s)\big) &= 0,\label{eq1}\\
\sin\theta(s)\,\big(t'(s)\,\kappa_\beta(t(s))-\cos\varphi\, \theta'(s)\big) &= -\sin\varphi\,\kappa_\gamma(s),\label{eq2}\\
-\sin\theta(s)\,\big(\sqrt{c}\,t'(s)+\sin\varphi\,\theta'(s)\big) &= \cos\varphi\,\kappa_\gamma(s).\label{eq3}
\end{align}
Since $M$ is a proper concircular surface then $\cos\theta(s)\neq0$, and from (\ref{eq1}) we get
\begin{equation}\label{p1.3}
\theta'(s)= t'(s)\big(\cos\varphi\,\kappa_\beta(t(s))-\sqrt{c}\,\sin\varphi\big).
\end{equation}
Finally, from a straightforward computation, we deduce that the curvature and torsion of the geodesic $\gamma$ are given by
\begin{align}
\kappa_\gamma(s) &= -\sin\theta(s)\,t'(s)\,(\sin\varphi\,\kappa_\beta(t(s))+\sqrt{c}\,\cos\varphi),\label{kg}\\
\tau_\gamma(s) &= \cos\theta(s)\,t'(s)\,(\sin\varphi\,\kappa_\beta(t(s))+\sqrt{c}\,\cos\varphi).\label{tg}
\end{align}
We have proved the following result.
\begin{proposition}\label{intp1}
A unit speed curve $\gamma(s)=X\big(t(s),z(s)\big)$, with $\kappa_\gamma>0$, is a geodesic of the nontrivial concircular surface $M$ if and only if there is a differentiable function $\theta$ such that equations (\ref{p1.1}), (\ref{p1.2}) and (\ref{p1.3}) hold. Moreover, the curvature and torsion of $\gamma$ are given by (\ref{kg}) and (\ref{tg}), respectively.
\end{proposition}

\begin{example}[Concircular helices ``of constant precession'']\label{ex0}
Let $\beta$ be a curve in a totally umbilical surface $S$ with curvature
\[
\kappa_\beta(t)=\frac{a}{\cos\varphi\sqrt{4-a^2t^2}}+\sqrt c\,\tan\varphi,\qquad t\in\Big(-\frac2a,\frac2a\Big),
\]
where $c\in\{0,1\}$ is the curvature of $S$, $a>0$ and $\varphi\in(0,\pi/2)$.
To find geodesics $\gamma$ of a concircular surface $M$ constructed on the curve $\beta$ (which will be concircular helices), we need to solve the ODE system given in Proposition \ref{intp1}. It is not difficult to check that a solution of (\ref{p1.1}), (\ref{p1.2}) and (\ref{p1.3}) is given by the following functions:
\begin{equation}\label{cpc}
\theta(s)=as,\qquad t(s)=-\frac{2}{a}\cos(a s),\qquad z(s)=\frac{1}{a}\sin(a s),
\end{equation}
for $s\in(0,\pi/a)$ (in general, $s\in(k\pi/a,(k+1)\pi/a)$ for a integer number $k$). In this case, the curvature and torsion of these concircular helices can be written by using (\ref{kg}) and (\ref{tg}) as follows,
\begin{align*}
\kappa_\gamma(s)&=-\sin{as}\left(a\tan\varphi+2\sqrt c\sec\varphi\,\sin(a s)\right),\\
\tau_\gamma(s)&=\cos{as}\left(a\tan\varphi+2\sqrt c\sec\varphi\,\sin(a s)\right).
\end{align*}
Hence, $\gamma$ is a curve of constant precession for $c=0$, \cite{Sco95}. In Figure \ref{fig.ex0} we show some examples for $c=1$, $\varphi=0.4\,\pi$, and several values of $a$.
\begin{figure}[htb]
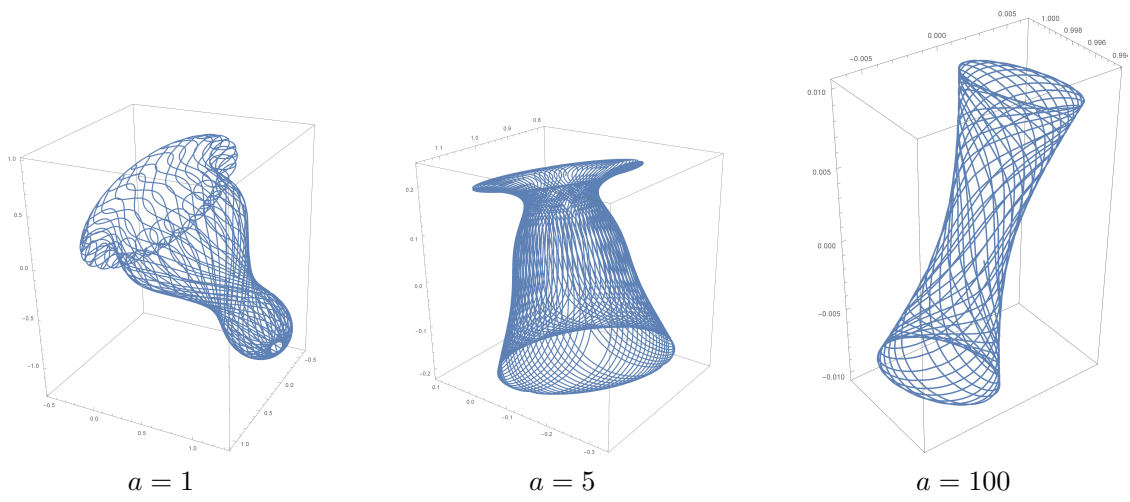
\centering
\def\ConcCurve#1;#2;#3;{%
 \includegraphics[width=.3\linewidth]{ConcCurve_a#1_f0.#2pi_s0#3.pdf}}
\def\suba#1#2{\raisebox{#1mm}{$\varphi=0.#2\pi$}}
\resizebox{\textwidth}{!}{%
\begin{tabular}{@{}c@{\kern10mm}c@{\kern10mm}c@{}}
\toprule
\ConcCurve1;4;90; & \ConcCurve5;4;40; & \ConcCurve100;4;0.5;\\
$a=1$ & $a=5$ & $a=100$\\
\bottomrule
\end{tabular}
}
\caption{\label{fig.ex0}Concircular helices ``of constant precession''}
\end{figure}
\end{example}

\subsection{A natural parametrization for the concircular surfaces}
\label{s:naturalpar}

Since the natural parametrization of a cylindrical surface, a conical surface or a helix surface is well known (see \cite{IT1, Chen03, LO16a}, respectively), let us consider the case when $M$ is a proper concircular surface.

Let $S\subset\R3$ be a sphere of radius $r$ centered at $p_0$. Take a unit speed curve $\beta:I\subset\R{}\to S\subset\R3$, an angle $\varphi\in(0,\pi/2)$, and consider the ruled surface $M$ parametrized by
\[
X(t,z)=\beta(t)+z\,\big(\cos\varphi\, N_\beta(t)+\sin\varphi\, \eta(t)\big).
\]
When $\kappa_\beta$ is constant (i.e. $\beta$ is a circle), $M$ is a cone and its striction line reduces to a single point (the vertex of the cone). Assume without loss of generality that $\kappa_\beta'\neq0$.

A straightforward computation shows that the striction line of $M$ is given by
\[
\alpha(t)=\beta(t)+\frac1{\cos\varphi\,\kappa_\beta(t)-\sqrt{c}\,\sin\varphi}\,\big(\cos\varphi\, N_\beta(t)+\sin\varphi\, \eta(t)\big).
\]
Hence the arclength parameter of $\alpha$ is
\[
u(t)=\frac{1}{\cos\varphi\,\kappa_\beta(t)-\sqrt{c}\,\sin\varphi},
\]
and the unit tangent vector field is given by
\(
T_\alpha(u(t))=\cos\varphi\, N_\beta(t)+\sin\varphi\, \eta(t).
\)
By taking derivative here we deduce $N_\alpha$ is parallel to $T_\beta$. Since $\<T_\beta(t),\beta(t)-p_0\>=0$ then
$\<N_\alpha(u(t)),\alpha(t)-p_0\>=0$, and this shows $\alpha$ is a rectifying curve.

A straightforward computation leads to
\[
\tau_\alpha(u(t))=-\kappa_\alpha(u(t))\big(\tan\varphi+\sqrt{c}\sec\varphi\,u(t)\big).
\]
Hence, a natural parametrization of $M$ is the following
\begin{equation}\label{parY2}
\Y(u,v)=\alpha(u)+(v-u)\,T_\alpha(u),\quad v\neq u,
\end{equation}
$\alpha$ being a rectifying curve. Since $\<B_\alpha,\alpha\>$ is constant for a rectifying curve and $B_\alpha$ is parallel to $N$, then the tangent surface $M$ to a rectifying curve $\alpha$ is a concircular surface.

Now, we will find the equations of their geodesic curves. Let $\gamma(s)=\Y(u(s),v(s))$ be a unit speed geodesic of the ruled surface $M$. Then
\[
T_\gamma(s)=u'(s)\,(v(s)-u(s))\,\kappa_\alpha(u(s))\,N_\alpha(u(s))+v'(s)\,T_\alpha(u(s)),
\]
so that there is a differentiable function $\omega$ such that
\begin{align*}
u'(s)\,(v(s)-u(s))\,\kappa_\alpha(u(s)) &= \sin\omega(s),\\
v'(s) &= \cos\omega(s).
\end{align*}
The function $\omega$ represents the angle between the geodesic $\gamma$ and the base curve $\alpha$. Bearing in mind the Frenet equations (\ref{FS-eq}) we obtain
\begin{align}
\kappa_\gamma(s)\,N_\gamma(s) &= -\sin\omega(s)\,\big(u'(s)\,\kappa_\alpha(u(s))+\omega'(s)\big)\,T_\alpha(u(s))\nonumber\\
&\kern10pt +\cos\omega(s)\,\big(u'(s)\,\kappa_\alpha(u(s))+\omega'(s)\big)\,N_\alpha(u(s))\nonumber\\
&\kern10pt +\sin\omega(s)\,u'(s)\,\tau_\alpha(u(s))\,B_\alpha(u(s)).\label{NgBa-}
\end{align}
Without loss of generality we can assume that
$N_\gamma(s)=N(u(s),v(s))=-B_\alpha(u(s))$,
that jointly with (\ref{NgBa-}) leads to
\(
\omega'(s) = -u'(s)\,\kappa_\alpha(u(s)).
\)
This yields the following equations for the curvature and torsion of the geodesic:
\begin{align}
\kappa_\gamma(s) &= -\sin\omega(s)\,u'(s)\,\tau_\alpha(u(s))= \sin\omega(s)\,\omega'(s)\,(\tau_\alpha/\kappa_\alpha)(u(s)),\label{kg2}\\
\tau_\gamma(s) &= \cos\omega(s)\,u'(s)\,\tau_\alpha(u(s))= -\cos\omega(s)\,\omega'(s)\,(\tau_\alpha/\kappa_\alpha)(u(s)).\label{tg2}
\end{align}
Hence, we have shown the following result.
\begin{proposition}\label{p3}
Let $M$ be a proper concircular surface parametrized by (\ref{parY2}). An arclength parametrized curve $\gamma(s)=\Y(u(s),v(s))$, with $\kappa_\gamma>0$, is a geodesic of $M$ if and only if there is a differentiable function $\omega$ such that
\begin{align}
u'(s)\,(v(s)-u(s))\,\kappa_\alpha(u(s)) &= \sin\omega(s),\label{p2.1}\\
v'(s) &= \cos\omega(s),\label{p2.2}\\
u'(s)\,\kappa_\alpha(u(s)) &= -\omega'(s).\label{p2.3}
\end{align}
Moreover, the curvature and torsion of $\gamma$ are given by (\ref{kg2}) and (\ref{tg2}), respectively.
\end{proposition}

\begin{theorem}\label{th1}
Let $\gamma(s)$ be an arclength parametrized curve fully immersed in $\R3$. If $\gamma$ is a proper concircular helix, then there exists a rectifying curve $\alpha$ such that $\gamma$ is (congruent to) a geodesic of the tangent surface to $\alpha$.
\end{theorem}
\begin{proof}
Let us consider the following functions:
\begin{align}
\omega(s)= & -\arccot\rho(s), \label{w teo}\\
u(s)= & -\frac{1}{m}\left(\kappa_\gamma(s)\,\frac{(1+\rho(s)^2)^{3/2}}{\rho'(s)}+n\right), \label{u teo}\\
v(s)= & u(s)-\frac{\sin\omega(s)}{\omega'(s)}, \label{v teo}
\end{align}
for a constant $n$. Here, $m$ and $\rho(s)$ are given in Theorem \ref{teoconc1}.

Let $\alpha$ be the rectifying curve determined by the curvature function
\begin{equation}\label{teo kalpha}
  \kappa_\alpha(t)=-\frac{\omega'(u^{-1}(t))}{u'(u^{-1}(t))}
\end{equation}
and whose torsion satisfies $(\tau_\alpha/\kappa_\alpha)(u)=m\,u+n$.

On the tangent surface to $\alpha$, let us take the curve $\tgamma(s)=\Y(u(s),v(s))$. From (\ref{teo kalpha}) we get (\ref{p2.3}), and then equation (\ref{v teo}) leads to (\ref{p2.1}).

By derivating (\ref{u teo}) and using (\ref{conc4}) we obtain
\[
u'(s)=\sqrt{1+\rho(s)^2}\,\frac{\rho''(s)}{\rho'(s)^2},
\]
and then we get
\[
v'(s)=\frac{\rho''(s)\sqrt{1+\rho(s)^2}}{\rho'(s)^2}+ \left(\frac{\sqrt{1+\rho(s)^2}}{\rho'(s)}\right)'
   =\frac{\rho(s)}{\sqrt{1+\rho^2(s)}}=\cos\omega(s).
\]
In conclusion, by Proposition \ref{p3} we get $\tgamma$ is a geodesic in the tangent surface to a rectifying curve, so it is a concircular helix.

On the other hand, we have
\begin{equation*}
  \frac{\tau_\alpha}{\kappa_\alpha}(u(s))=m\,u(s)+n=-\kappa_\gamma(s)\,\frac{(1+\rho(s)^2)^{3/2}}{\rho'(s)}.
\end{equation*}
Since the curvature and torsion of $\tilde\gamma$ satisfy (\ref{kg2}) y (\ref{tg2}), respectively, we get
\begin{align*}
  \kappa_{\tgamma}(s)= &  \frac{-1}{\sqrt{1+\rho(s)^2}}\,\frac{\rho'(s)}{1+\rho(s)^2}\,\frac{-(1+\rho(s)^2)^{3/2}}{\rho'(s)}\,\kappa_\gamma(s)=\kappa_\gamma(s),\\
  \tau_{\tgamma}(s)= &  -\frac{\rho(s)}{\sqrt{1+\rho(s)^2}}\,\frac{\rho'(s)}{1+\rho(s)^2}\,\frac{-(1+\rho(s)^2)^{3/2}}{\rho'(s)}\,\kappa_\gamma(s)=\tau_\gamma(s),
\end{align*}
which shows that $\gamma$ and $\tgamma$ are congruent curves.
\end{proof}

\acknowledgment{This research is part of the grant PID2021-124157NB-I00, funded by MCIN/ AEI/ 10.13039/ 501100011033/ ``ERDF A way of making Europe''.} %If necessary

\end{document}